\documentclass[11pt]{amsart}

\author[C.~Sanna]{Carlo Sanna$^\dagger$}
\thanks{$\dagger\,$C.~Sanna is a member of GNSAGA of INdAM and of CrypTO, the group of Cryptography and Number~Theory of Politecnico di Torino}
\address{\parbox{\linewidth}{
Politecnico di Torino, Department of Mathematical Sciences\\
Corso Duca degli Abruzzi 24, 10129 Torino, Italy\\[-8pt]}}
\email{carlo.sanna.dev@gmail.com}


\title{On the least common multiple of random $q$-integers}

\usepackage{bm}
\usepackage{amsmath}
\usepackage{amssymb}
\usepackage{amsthm}
\usepackage{geometry}
\usepackage{color}
\usepackage{hyperref}
\usepackage{enumerate}
\hypersetup{colorlinks=true}
\usepackage{cite}

\newtheorem{theorem}{Theorem}[section]

\newtheorem{lemma}[theorem]{Lemma}
\theoremstyle{remark}
\newtheorem{remark}{Remark}[section]

\newcommand{\lcm}{\operatorname{lcm}}
\newcommand{\Li}{\operatorname{Li}}

\begin{document}

\begin{abstract}
For every positive integer $n$ and for every $\alpha \in [0, 1]$, let $\mathcal{B}(n, \alpha)$ denote the probabilistic model in which a random set $\mathcal{A} \subseteq \{1, \dots, n\}$ is constructed by picking independently each element of $\{1, \dots, n\}$ with probability $\alpha$.
Cilleruelo, Ru\'{e}, \v{S}arka, and Zumalac\'{a}rregui proved an almost sure asymptotic formula for the logarithm of the least common multiple of the elements of $\mathcal{A}$.

Let $q$ be an indeterminate and let $[k]_q := 1 + q + q^2 + \cdots + q^{k-1} \in \mathbb{Z}[q]$ be the $q$-analog of the positive integer $k$.
We determine the expected value and the variance of $X := \deg \lcm\!\big([\mathcal{A}]_q\big)$, where $[\mathcal{A}]_q := \big\{[k]_q : k \in \mathcal{A}\big\}$.
Then we prove an almost sure asymptotic formula for $X$, which is a $q$-analog of the result of Cilleruelo~et~al.
\end{abstract}

\maketitle

\section{Introduction}

For every positive integer $n$ and every $\alpha \in [0, 1]$, let $\mathcal{B}(n, \alpha)$ denote the probabilistic model in which a random set $\mathcal{A} \subseteq \{1, \dots, n\}$ is constructed by picking independently each element of $\{1, \dots, n\}$ with probability $\alpha$.
Cilleruelo, Ru\'{e}, \v{S}arka, and Zumalac\'{a}rregui~\cite{MR3239153} studied the least common multiple $\lcm(\mathcal{A})$ of the elements of $\mathcal{A}$ and proved the following result (see~\cite{MR4009436} for a more precise version, and~\cite{MR3649012, MR3640773, MR4091939, Preprint0, MR3991415, Preprint1} for other results of a similar flavor).

\begin{theorem}\label{thm:cilleruelo}
Let $\mathcal{A}$ be a random set in $\mathcal{B}(n, \alpha)$.
Then, as $\alpha n \to +\infty$, we have
\begin{equation*}
\log \lcm(\mathcal{A}) \sim \frac{\alpha \log(1/\alpha)}{1 - \alpha} \cdot n ,
\end{equation*}
with probability $1 - o(1)$, where the factor involving $\alpha$ is meant to be equal to $1$ for $\alpha = 1$.
\end{theorem}

Let $q$ be an indeterminate.
The \emph{$q$-analog} of a positive integer $k$ is defined by 
\begin{equation*}
[k]_q := 1 + q + q^2 + \cdots + q^{k - 1} \in \mathbb{Z}[q] .
\end{equation*}
The $q$-analogs of many other mathematical objects (factorial, binomial coefficients, hypergeometric series, derivative, integral...) have been extensively studied, especially in Analysis and Combinatorics~\cite{MR1865777, MR858826}.
For every set $\mathcal{S}$ of positive integers, let $[\mathcal{S}]_q := \big\{[k]_q : k \in \mathcal{S}\big\}$.

The aim of this paper is to study the least common multiple of the elements of $[\mathcal{A}]_q$ for a random set $\mathcal{A}$ in $\mathcal{B}(n, \alpha)$.
Our main results are the following:

\begin{theorem}\label{thm:expectation}
Let $\mathcal{A}$ be a random set in $\mathcal{B}(n, \alpha)$ and put $X := \deg \lcm\!\big([\mathcal{A}]_q\big)$.
Then, for every integer $n \geq 2$ and every $\alpha \in [0,1]$, we have
\begin{equation}\label{equ:expectation}
\mathbb{E}[X] = \frac{3}{\pi^2} \cdot \frac{\alpha \Li_2(1 - \alpha)}{1 - \alpha} \cdot n^2 + O\!\left(\alpha n (\log n)^2 \right) ,
\end{equation}
where $\Li_2(z) := \sum_{k=1}^\infty z^k / k^2$ is the dilogarithm and the factor involving $\alpha$ is meant to be equal to $1$ when $\alpha = 1$.
In particular,
\begin{equation*}
\mathbb{E}[X] \sim \frac{3}{\pi^2} \cdot \frac{\alpha \Li_2(1 - \alpha)}{1 - \alpha} \cdot n^2 ,
\end{equation*}
as $n \to +\infty$, uniformly for $\alpha \in {(0,1]}$.
\end{theorem}

\begin{theorem}\label{thm:variance}
Let $\mathcal{A}$ be a random set in $\mathcal{B}(n, \alpha)$ and put $X := \deg \lcm\!\big([\mathcal{A}]_q\big)$.
Then there exists a function $\mathrm{v} : {(0,1)} \to \mathbb{R}^+$ such that, as $\alpha n / \big((\log n)^3 \log \log n\big) \to +\infty$, we have
\begin{equation}\label{equ:variance_asymp}
\mathbb{V}[X] = (\mathrm{v}(\alpha) + o(1)) \, n^3 .
\end{equation}
Moreover, the upper bound
\begin{equation}\label{equ:variance_upper}
\mathbb{V}[X] \ll \alpha n^3 ,
\end{equation}
holds for every positive integer $n$ and every $\alpha \in [0, 1]$.
\end{theorem}

As a consequence of Theorem~\ref{thm:expectation} and Theorem~\ref{thm:variance}, we obtain the following $q$-analog of Theorem~\ref{thm:cilleruelo}.

\begin{theorem}\label{thm:Aq}
Let $\mathcal{A}$ be a random set in $\mathcal{B}(n, \alpha)$.
Then, as $\alpha n \to +\infty$, we have
\begin{equation*}
\deg \lcm\!\big([\mathcal{A}]_q\big) \sim \frac{3}{\pi^2} \cdot \frac{\alpha \Li_2(1 - \alpha)}{1 - \alpha} \cdot n^2 ,
\end{equation*}
with probability $1 - o(1)$, where the factor involving $\alpha$ is meant to be equal to $1$ for $\alpha = 1$.
\end{theorem}

We remark that in Theorem~\ref{thm:Aq} the condition $\alpha n \to +\infty$ is necessary.
Indeed, if $\alpha n \leq C$, for some constant $C > 0$, then 
\begin{equation*}
\mathbb{P}[\mathcal{A} = \varnothing] = (1 - \alpha)^n \geq \left(1 - \frac{C}{n}\right)^n \to e^C
\end{equation*}
as $n \to +\infty$, and so no (nontrivial) asymptotic formula for $\deg \lcm\!\big([\mathcal{A}]_q\big)$ can hold with probability $1 - o(1)$.

\section{Notation}

We employ the Landau--Bachmann ``Big Oh'' and ``little oh'' notations $O$ and $o$, as well as the associated Vinogradov symbol $\ll$, with their usual meanings.
Any dependence of the implied constants is explicitly stated or indicated with subscripts.
For real random variables $X$ and $Y$, we say that ``$X \sim Y$ with probability $1 - o(1)$'' if $\mathbb{P}\big[\,|X - Y| > \varepsilon|Y|\,\big] = o_\varepsilon(1)$ for every $\varepsilon > 0$.
We let $(a,b)$ and $[a,b]$ denote the greatest common divisor and the least common multiple, respectively, of two integers $a$ and $b$.
As usual, we write $\varphi(n)$, $\mu(n)$, $\tau(n)$, and $\sigma(n)$, for the Euler totient function, the M\"obius function, the number of divisors, and the sum of divisors, of a positive integer $n$, respectively.

\section{Preliminaries}

In this section we collect some preliminary results needed in later arguments.

\begin{lemma}\label{lem:tau}
We have
\begin{equation*}
\sum_{m \,\leq\, x} \tau(m) \ll x \log x ,
\end{equation*}
for every $x \geq 2$.
\end{lemma}
\begin{proof}
See, e.g.,~\cite[Theorem~3.2]{MR3363366}.
\end{proof}

\begin{lemma}\label{lem:lcme1e2}
We have
\begin{equation*}
\sum_{[e_1\!,\, e_2] \,>\, x} \frac1{e_1 e_2 [e_1, e_2]} \ll \frac{\log x}{x}
\end{equation*}
for every $x \geq 2$.
\end{lemma}
\begin{proof}
From Lemma~\ref{lem:tau} and partial summation, it follows that
\begin{equation*}
\sum_{m \,>\, x} \frac{\tau(m)}{m^2} \ll \frac{\log x}{x} .
\end{equation*}
Let $e := (e_1, e_2)$ and $e_1^\prime := e_i / e$ for $i=1,2$.
Then we have
\begin{align*}
\sum_{[e_1\!,\, e_2] \,>\, x} \frac1{e_1 e_2 [e_1, e_2]} &\leq \sum_{e \,\geq\, 1} \frac1{e^3} \sum_{e_1^\prime e_2^\prime \,>\, x / e} \frac1{(e_1^\prime e_2^\prime)^2} = \sum_{e \,\geq\, 1} \frac1{e^3} \sum_{m \,>\, x / e} \frac{\tau(m)}{m^2} \\
&\ll \sum_{e \,\leq\, x / 2} \frac1{e^3} \frac{\log(x/e)}{x/e} + \sum_{e \,>\, x / 2} \frac1{e^3} \ll \frac{\log x}{x} + \frac1{x^2} \ll \frac{\log x}{x} ,
\end{align*}
as desired.
\end{proof}

Let us define
\begin{equation*}
\Phi(x) := \sum_{n \,\leq\, x} \varphi(n) \quad\text{ and }\quad \Phi(a_1, a_2; x) := \sum_{n \,\leq\, x} \varphi(a_1 n)\, \varphi(a_2 n) ,
\end{equation*}
for every $x \geq 1$ and for all positive integers $a_1, a_2$.

\begin{lemma}\label{lem:Phi}
For every $x \geq 2$, we have
\begin{equation*}
\Phi(x) = \frac{3}{\pi^2} \, x^2 + O(x \log x) ,
\end{equation*}
\end{lemma}
\begin{proof}
See, e.g.,~\cite[Theorem~3.4]{MR3363366}.
\end{proof}

\begin{lemma}\label{lem:Phia1a2}
We have
\begin{equation}\label{equ:Phia1a2}
\Phi(a_1, a_2; x) = C_1(a_1, a_2) \, x^3 + O\big(\sigma(a_1 a_2) \,x^2 (\log x)^2\big) ,
\end{equation}
for every $x \geq 2$, where
\begin{equation}\label{equ:C1a1a2}
C_1(a_1, a_2) := \frac{a_1 a_2}{3}\sum_{d_1\!,\, d_2 \,\geq\, 1} \frac{\mu(d_1)\mu(d_2)}{d_1 d_2 [d_1 / (a_1, d_1), d_2 / (a_2, d_2)]}
\end{equation}
and the series is absolutely convergent.
\end{lemma}
\begin{proof}
From the identity $\varphi(n) / n = \sum_{d \,\mid\;\;\!\!\! n} \mu(d) / d$, it follows that
\begin{align*}
\sum_{n \,\leq\, x} \frac{\varphi(a_1 n)}{a_1 n} \,\frac{\varphi(a_2 n)}{a_2 n} &= \sum_{n \,\leq\, x} \left(\sum_{d_1 \,\mid\, a_1 n} \frac{\mu(d_1)}{d_1} \sum_{d_2 \,\mid\, a_2 n} \frac{\mu(d_2)}{d_2} \right) \\
 &= \sum_{\substack{d_1 \,\leq\, a_1 x \\[2pt] d_2 \,\leq\, a_2 x}} \frac{\mu(d_1)}{d_1} \, \frac{\mu(d_2)}{d_2}\, \#\big\{n \leq x : d_1 \mid a_1 n \text{ and } d_2 \mid a_2 n \big\} \\
 &= \sum_{[d_1^\prime\!,\, d_2^\prime] \,\leq\, x} \frac{\mu(d_1)}{d_1} \, \frac{\mu(d_2)}{d_2} \left(\frac{x}{[d_1^\prime, d_2^\prime]} + O(1)\right) ,
\end{align*}
where $d_i^\prime := d_i / (a_i, d_i)$ for $i = 1, 2$.
On the one hand, we have
\begin{equation*}
\sum_{[d_1^\prime\!,\, d_2^\prime] \,\leq\, x} \frac1{d_1 d_2} \leq \sum_{c \,\mid\, a_1 a_2} \frac1{c} \sum_{e_1\!,\, e_2 \,\leq\, x} \frac1{e_1 e_2} \ll \frac{\sigma(a_1 a_2)}{a_1 a_2} \, (\log x)^2 .
\end{equation*}
On the other hand, thanks to Lemma~\ref{lem:lcme1e2}, we have
\begin{equation*}
\sum_{[d_1^\prime\!,\, d_2^\prime] \,>\, x} \frac1{d_1 d_2 [d_1^\prime, d_2^\prime]} \leq \sum_{c \,\mid\, a_1 a_2} \frac1{c} \sum_{[e_1\!,\, e_2] \,>\, x} \frac1{e_1 e_2 [e_1, e_2]} \ll \frac{\sigma(a_1 a_2)}{a_1 a_2} \,\frac{\log x}{x} ,
\end{equation*}
which, in particular, implies that series~\eqref{equ:C1a1a2} is absolutely convergent.
Therefore, letting $C_0(a_1, a_2) := 3 C_1(a_1, a_2) / (a_1 a_2)$, we obtain
\begin{align}\label{equ:Phia1a2norm}
\sum_{n \,\leq\, x} \frac{\varphi(a_1 n)}{a_1 n} \,\frac{\varphi(a_2 n)}{a_2 n} &= \left(\!C_0(a_1, a_2) + O\!\!\left(\sum_{[d_1^\prime\!,\, d_2^\prime] \,>\, x} \frac1{d_1 d_2 [d_1, d_2]}\right)\!\!\right) x + O\!\!\left(\sum_{[d_1^\prime\!,\, d_2^\prime] \,\leq\, x} \frac1{d_1 d_2}\right) \\
 &= C_0(a_1, a_2) \, x + O\!\left(\frac{\sigma(a_1 a_2)}{a_1 a_2} \,(\log x)^2\right) . \nonumber
\end{align}
Now~\eqref{equ:Phia1a2} follows easily from~\eqref{equ:Phia1a2norm} by partial summation.
\end{proof}

\begin{remark}\label{rmk:C1a1a2}
The obvious bound $\varphi(m) \leq m$ yields $C_1(a_1, a_2) \leq a_1 a_2 / 3$ (which is not so obvious from~\eqref{equ:C1a1a2}).
\end{remark}

The following lemma is an easy inequality that will be useful later.

\begin{lemma}\label{lem:Bernoulli}
It holds $1 - (1 - x)^k \leq k x$, for all $x \in [0, 1]$ and all integers $k \geq 0$.
\end{lemma}
\begin{proof}
The claim is $(1 + (-x))^k \geq 1 + k(-x)$, which follows from Bernoulli's inequality.
\end{proof}

\section{Proofs}

Henceforth, let $\mathcal{A}$ be a random set in $\mathcal{B}(n, \alpha)$, let $[\mathcal{A}]_q$ be its $q$-analog, and put $L := \lcm\!\big([\mathcal{A}]_q\big)$ and $X := \deg L$.
For every positive integer $d$, let us define
\begin{equation*}
I_{\mathcal{A}}(d) := \begin{cases} 1 & \text{ if } d \mid k \text{ for some } k \in \mathcal{A}; \\ 0 & \text{ otherwise.} \end{cases}
\end{equation*}
The following lemma gives a formula for $X$ in terms of $I_{\mathcal{A}}$ and the Euler function.

\begin{lemma}\label{lem:XsumphiIA}
We have
\begin{equation}\label{equ:XsumphiIA}
X = \sum_{1 \,<\, d \,\leq\, n} \varphi(d)\, I_{\mathcal{A}}(d) .
\end{equation}
\end{lemma}
\begin{proof}
For every positive integer $k$, it holds
\begin{equation*}
[k]_q = \frac{q^k - 1}{q - 1} = \prod_{\substack{d \,\mid\!\; k \\ d \,>\, 1}} \Phi_d(q) ,
\end{equation*}
where $\Phi_d(q)$ is the $d$th cyclotomic polynomials.
Since, as it is well known, every cyclotomic polynomial is irreducible over $\mathbb{Q}$, it follows that $L$ is the product of the polynomials $\Phi_d(q)$ such that $d > 1$ and $d \mid k$ for some $k \in \mathcal{A}$.
Finally, the equality $\deg\!\big(\Phi_d(q)\big) = \varphi(d)$ and the definition of $I_{\mathcal{A}}$ yield~\eqref{equ:XsumphiIA}.
\end{proof}

Let $\beta := 1 - \alpha$.
The next lemma provides two expected values involving $I_{\mathcal{A}}$.

\begin{lemma}\label{lem:EIAd}
For all positive integers $d, d_1, d_2$, we have
\begin{equation}\label{equ:EIAd}
\mathbb{E}\big[I_\mathcal{A}(d)\big] = 1 - \beta^{\lfloor n / d\rfloor}
\end{equation}
and
\begin{align*}
\mathbb{E}\big[I_\mathcal{A}(d_1)I_\mathcal{A}(d_2)\big] = 1 - \beta^{\lfloor n / d_1 \rfloor}  - \beta^{\lfloor n / d_2 \rfloor} + \beta^{\lfloor n / d_1 \rfloor + \lfloor n / d_2 \rfloor - \lfloor n / [d_1\!,\, d_2] \rfloor} .
\end{align*}
\end{lemma}
\begin{proof}
On the one hand, by the definition of $I_{\mathcal{A}}$, we have
\begin{equation*}
\mathbb{E}\big[I_{\mathcal{A}}(d)\big] = \mathbb{P}\big[\exists k \in \mathcal{A} : d \mid k\big] = 1 - \mathbb{P}\left[\bigwedge_{m \,\leq\, \lfloor n / d\rfloor} (dm \notin \mathcal{A})\right] = 1 - \beta^{\lfloor n / d \rfloor} ,
\end{equation*}
which is~\eqref{equ:EIAd}.
On the other hand, by linearity of the expectation and by~\eqref{equ:EIAd}, we have
\begin{align*}
\mathbb{E}\big[I_\mathcal{A}(d_1)I_\mathcal{A}(d_2)\big] &= \mathbb{E}\big[I_\mathcal{A}(d_1) + I_\mathcal{A}(d_2) - 1 + \big(1 - I_\mathcal{A}(d_1)\big)\big(1 - I_\mathcal{A}(d_2)\big)\big] \\
 &= \mathbb{E}\big[I_\mathcal{A}(d_1)\big] + \mathbb{E}\big[I_\mathcal{A}(d_2)\big] - 1 + \mathbb{E}\big[\big(1 - I_\mathcal{A}(d_1)\big)\big(1 - I_\mathcal{A}(d_2)\big)\big] \\
 &= 1 - \beta^{\lfloor n / d_1 \rfloor}  - \beta^{\lfloor n / d_2 \rfloor} + \mathbb{E}\big[\big(1 - I_\mathcal{A}(d_1)\big)\big(1 - I_\mathcal{A}(d_2)\big)\big] ,
\end{align*}
where the last expected value can be computed as
\begin{align*}
\mathbb{E}\big[\big(1 - I_\mathcal{A}(d_1)\big)&\big(1 - I_\mathcal{A}(d_2)\big)\big] = \mathbb{P}\big[\forall k \in \mathcal{A} : d_1 \nmid k \text{ and } d_2 \nmid k\big] \\
 &= \mathbb{P}\left[\bigwedge_{\substack{k \,\leq\, n \\ d_1 \,\mid\, k \text{ or } d_2 \,\mid\, k}}(k \notin \mathcal{A})\right] = \beta^{\lfloor n / d_1 \rfloor + \lfloor n / d_2 \rfloor - \lfloor n / [d_1\!,\, d_2] \rfloor} ,
\end{align*}
and second claim follows.
\end{proof}

We are ready to compute the expected value of $X$.

\begin{proof}[Proof of Theorem~\ref{thm:expectation}]
From Lemma~\ref{lem:XsumphiIA} and Lemma~\ref{lem:EIAd}, it follows that
\begin{equation}\label{equ:EX1}
\mathbb{E}[X] = \sum_{1 \,<\, d \,\leq\, n} \varphi(d)\, \mathbb{E}\big[I_{\mathcal{A}}(d)\big] = \sum_{1 \,<\, d \,\leq\, n} \varphi(d) \big(1 - \beta^{\lfloor n / d \rfloor}\big) .
\end{equation}
Moreover, since $\lfloor n / d \rfloor = j$ if and only if $n / (j + 1) < d \leq n / j$, we get that
\begin{align}\label{equ:EX2}
\sum_{d \,\leq\, n} \varphi(d) \big(1 - \beta^{\lfloor n / d \rfloor}\big) &= \sum_{j \,\leq\, n} (1 - \beta^j) \sum_{n / (j + 1) \,<\, d \,\leq\, n / j} \varphi(d) \\
&= \sum_{j \,\leq\, n} (1 - \beta^j) \!\left(\Phi\!\left(\frac{n}{j}\right) - \Phi\!\left(\frac{n}{j + 1}\right)\right) \nonumber\\
&= \alpha \sum_{j \,\leq\, n} \beta^{j - 1} \Phi\!\left(\frac{n}{j}\right) \nonumber\\
&= \frac{3}{\pi^2} \cdot \alpha \sum_{j \,\leq\, n} \frac{\beta^{j-1}}{j^2} \cdot n^2 + O\!\left(\alpha \sum_{j \,\leq\, n} \frac{n}{j}\log\!\left(\frac{n}{j}\right) \right) \nonumber\\
&= \frac{3}{\pi^2} \cdot \frac{\alpha \Li_2(1 - \alpha)}{1 - \alpha} \cdot n^2 + O \big(\alpha n (\log n)^2\big) , \nonumber
\end{align}
where we used Lemma~\ref{lem:Phi}.
Putting together~\eqref{equ:EX1} and~\eqref{equ:EX2}, and noting that, by Lemma~\ref{lem:Bernoulli}, the addend of~\eqref{equ:EX2} corresponding to $d = 1$ is $1 - \beta^n = O(\alpha n)$, we get~\eqref{equ:expectation}.
The proof is complete.
\end{proof}

Now we consider the variance of $X$.

\begin{proof}[Proof of Theorem~\ref{thm:variance}]
From Lemma~\ref{lem:XsumphiIA} and Lemma~\ref{lem:EIAd}, it follows that
\begin{align}\label{equ:VX}
\mathbb{V}[X] &= \mathbb{E}\big[X^2\big] - \mathbb{E}[X]^2 \\
 &= \sum_{1 \,<\, d_1\!,\, d_2 \,\leq\, n} \varphi(d_1)\,\varphi(d_2) \Big(\mathbb{E}\big[I_{\mathcal{A}}(d_1)\, I_{\mathcal{A}}(d_2)\big] - \mathbb{E}\big[I_{\mathcal{A}}(d_1)\big]\,\mathbb{E}\big[I_{\mathcal{A}}(d_2)\big]\Big) \nonumber \\
 &= \sum_{1 \,<\, d_1\!,\, d_2 \,\leq\, n} \varphi(d_1)\,\varphi(d_2) \, \beta^{\lfloor n / d_1 \rfloor + \lfloor n / d_2 \rfloor - \lfloor n / [d_1, d_2] \rfloor} \big(1 - \beta^{\lfloor n / [d_1, d_2] \rfloor} \big) . \nonumber 
\end{align}
Let us define
\begin{equation*}
V_n(\alpha) := \frac1{n^3}\sum_{d_1\!,\, d_2 \,\leq\, n} \varphi(d_1)\,\varphi(d_2) \, \beta^{\lfloor n / d_1 \rfloor + \lfloor n / d_2 \rfloor - \lfloor n / [d_1, d_2] \rfloor} \big(1 - \beta^{\lfloor n / [d_1, d_2] \rfloor} \big) .
\end{equation*}
Clearly, we have
\begin{equation*}
V_n(\alpha) - \frac{\mathbb{V}[X]}{n^3} \ll \frac1{n^3}\sum_{d \,\leq\, n} \varphi(d) \, \beta^{n} \big(1 - \beta^{\lfloor n / d \rfloor} \big) \leq \frac1{n^3}\sum_{d \,\leq\, n} d \ll \frac1{n}.
\end{equation*}
Hence, in order to prove~\eqref{equ:variance_asymp}, it sufficies to show that $V_n(\alpha) = \mathrm{v}(\alpha) + o(1)$.

Let $d := (d_1, d_2)$ and $a_i := d_i / d$ for $i=1,2$.
Then, for all positive integers $j_0, j_1, j_2$, an easy computation shows that the equalities
\begin{equation*}
j_1 = \left\lfloor \frac{n}{d_1}\right\rfloor , \quad
j_2 = \left\lfloor \frac{n}{d_2}\right\rfloor , \quad
j_3 = \left\lfloor \frac{n}{[d_1, d_2]}\right\rfloor ,
\end{equation*}
are equivalent to
\begin{equation*}
\rho_1(\bm{a}, \bm{j})\, n < d \leq \rho_2(\bm{a}, \bm{j})\, n ,
\end{equation*}
where
\begin{equation*}
\rho_1(\bm{a}, \bm{j}) := \max\!\left(\frac1{a_1(j_1 + 1)}, \frac1{a_2(j_2 + 1)}, \frac1{a_1 a_2 (j_3 + 1)}\right)
\end{equation*}
and
\begin{equation*}
\rho_2(\bm{a}, \bm{j}) := \min\!\left(\frac1{a_1 j_1}, \frac1{a_2 j_2}, \frac1{a_1 a_2 j_3}\right) .
\end{equation*}
Therefore, letting
\begin{equation*}
\mathcal{S}_n := \big\{(\bm{a}, \bm{j}) \in \mathbb{N}^5 : (a_1, a_2) = 1,\; \exists d \in \mathbb{N} \text{ s.t. }\! \rho_1(\bm{a}, \bm{j})\, n < d \leq \rho_2(\bm{a}, \bm{j})\, n \big\}
\end{equation*}
and
\begin{equation*}
S(\bm{a}, \bm{j}; n) := \frac1{n^3} \sum_{\rho_1(\bm{a},\, \bm{j})\, n \,<\, d \,\leq\, \rho_2(\bm{a},\, \bm{j})\, n} \varphi(a_1 d) \, \varphi(a_2 d) ,
\end{equation*}
we have
\begin{equation*}
V_n(\alpha) = \sum_{(\bm{a},\, \bm{j}) \,\in\, \mathcal{S}_n} \beta^{j_1 + j_2 - j_3} (1 - \beta^{j_3}) \,S(\bm{a}, \bm{j}; n) .
\end{equation*}
Now let us define
\begin{equation}\label{equ:valpha}
\mathrm{v}(\alpha) := \sum_{(\bm{a},\, \bm{j}) \,\in\, \mathcal{S}_\infty} \beta^{j_1 + j_2 - j_3} (1 - \beta^{j_3}) \, D(\bm{a}, \bm{j}) ,
\end{equation}
where
\begin{equation*}
\mathcal{S}_\infty := \bigcup_{m \,\geq\, 1} \mathcal{S}_m = \big\{(\bm{a}, \bm{j}) \in \mathbb{N}^5 : (a_1, a_2) = 1,\, \rho_1(\bm{a}, \bm{j}) < \rho_2(\bm{a}, \bm{j}) \big\}
\end{equation*}
and
\begin{equation*}
D(\bm{a}, \bm{j}) := C_1(a_1, a_2) \big(\rho_2(\bm{a},\, \bm{j})^3 - \rho_1(\bm{a},\, \bm{j})^3\big) .
\end{equation*}
The convergence of series~\eqref{equ:valpha} follows easily from Remark~\ref{rmk:C1a1a2}, $\rho_2(\bm{a}, \bm{j}) \leq 1 / (a_1 a_2 j_3)$, and the fact that $\min(j_1, j_2) \geq j_3$ for all $(\bm{a}, \bm{j}) \in \mathcal{S}_\infty$.

Thanks to Lemma~\ref{lem:Phia1a2}, for each $(\bm{a}, \bm{j}) \in \mathcal{S}_n$ we have
\begin{equation*}
S(\bm{a}, \bm{j}; n) = D(\bm{a}, \bm{j}) + O\!\left(\sigma(a_1 a_2) \,\rho_2(\bm{a}, \bm{j})^2 \cdot \frac{(\log n)^2}{n}\right) .
\end{equation*}
Consequently, we get that
\begin{equation}\label{equ:Vnalphabound}
V_n(\alpha) = \mathrm{v}(\alpha) - \Sigma_1 + O\!\left(\Sigma_2 \cdot \frac{(\log n)^2}{n}\right) ,
\end{equation}
where
\begin{equation*}
\Sigma_1 := \sum_{(\bm{a},\, \bm{j}) \,\in\, \mathcal{S}_\infty \!\setminus \mathcal{S}_n} \beta^{j_1 + j_2 - j_3} (1 - \beta^{j_3}) \, D(\bm{a}, \bm{j})
\end{equation*}
and
\begin{equation*}
\Sigma_2 := \sum_{(\bm{a},\, \bm{j}) \,\in\, \mathcal{S}_n} \beta^{j_1 + j_2 - j_3} (1 - \beta^{j_3}) \, \sigma(a_1 a_2) \,\rho_2(\bm{a}, \bm{j})^2 .
\end{equation*}

If $(\bm{a}, \bm{j}) \in \mathcal{S}_\infty$ then, as we already noticed, $\min(j_1, j_2) \geq j_3$ and, moreover,
\begin{equation*}
\frac{j_2}{j_3 + 1} < a_1 < \frac{j_2 + 1}{j_3} \quad\text{ and }\quad \frac{j_1}{j_3 + 1} < a_2 < \frac{j_1 + 1}{j_3} .
\end{equation*}
Hence, we have
\begin{align}\label{equ:sumbetaj1j2j3}
\sum_{(\bm{a},\, \bm{j}) \,\in\, \mathcal{S}_\infty}& \frac{\beta^{j_1 + j_2 - j_3} (1 - \beta^{j_3})}{a_1 a_2 j_3^2} \leq \sum_{j_3 \,\geq\, 1} \frac{1 - \beta^{j_3}}{j_3^2} \sum_{j_1,\, j_2 \,\geq\, j_3} \beta^{j_1 + j_2 - j_3} \sum_{\substack{j_2 / (j_3 + 1) \,<\, a_1 \,<\, (j_2 + 1) / j_3 \\ j_1 / (j_3 + 1) \,<\, a_2 \,<\, (j_1 + 1) / j_3}} \frac1{a_1 a_2} \\
&\ll \sum_{j_3 \,\geq\, 1} \frac{1 - \beta^{j_3}}{j_3^2} \sum_{j_1,\, j_2 \,\geq\, j_3} \beta^{j_1 + j_2 - j_3} = \frac1{\alpha^2}\sum_{j \,\geq\, 1} \frac{(1 - \beta^j)\beta^j}{j^2} \nonumber \\
&\leq \frac1{\alpha} \sum_{j \,\leq\, 1 / \alpha} \frac1{j} + \frac1{\alpha^2}\sum_{j \,>\, 1 / \alpha} \frac1{j^2} \ll \frac{\log(1 / \alpha) + 1}{\alpha} , \nonumber 
\end{align}
where we used the inequality $1 - \beta^j \leq \alpha j$, which follows from Lemma~\ref{lem:Bernoulli}.

If $(\bm{a}, \bm{j}) \in \mathcal{S}_\infty \setminus \mathcal{S}_n$ then $\big(\rho_2(\bm{a}, \bm{j}) - \rho_1(\bm{a}, \bm{j})\big) n < 1$ and consequently, also by Remark~\ref{rmk:C1a1a2},
\begin{equation}\label{equ:Daj_bound}
D(\bm{a}, \bm{j}) \ll a_1 a_2 \big(\rho_2^3 - \rho_1^3\big) = a_1 a_2 \big(\rho_1^2 + \rho_1 \rho_2 + \rho_2^2\big)(\rho_2 - \rho_1) \ll \frac{a_1 a_2 \rho_2^2}{n} \leq \frac1{a_1 a_2 j_3^2 n} ,
\end{equation}
where, for brevity, we wrote $\rho_i := \rho_i(\bm{a}, \bm{j})$ for $i=1,2$.

On the one hand, from~\eqref{equ:sumbetaj1j2j3} and~\eqref{equ:Daj_bound} it follows that
\begin{equation}\label{equ:Sigma1bound}
\Sigma_1 \ll \frac{\log(1 / \alpha) + 1}{\alpha n} = o(1) ,
\end{equation}
as $\alpha n / \!\big((\log n)^3 \log \log n)\big) \to +\infty$ (actually, $\alpha n / \!\log n \to +\infty$ is sufficient).
On the other hand, from~\eqref{equ:sumbetaj1j2j3} and the inequality $\sigma(m) \leq m \log \log m$ (see, e.g.,~\cite[Theorem~5.7]{MR3363366}) it follows that
\begin{equation}\label{equ:Sigma2bound}
\Sigma_2 \leq \sum_{(\bm{a},\, \bm{j}) \,\in\, \mathcal{S}_n} \frac{\beta^{j_1 + j_2 - j_3} (1 - \beta^{j_3})}{a_1 a_2 j_3^2} \cdot \frac{\sigma(a_1 a_2)}{a_1 a_2} \ll \frac{(\log(1/\alpha) + 1) \log \log n}{\alpha} = o\!\left(\frac{n}{(\log n)^2}\right) ,
\end{equation}
as $\alpha n / \big((\log n)^3 \log \log n\big) \to +\infty$.

At this point, putting together~\eqref{equ:Vnalphabound},~\eqref{equ:Sigma1bound}, and~\eqref{equ:Sigma2bound}, we obtain $V_n(\alpha) = \mathrm{v}(\alpha) + o(1)$.
The proof of~\eqref{equ:variance_asymp} is complete.

It remains only to prove the upper bound~\eqref{equ:variance_upper}.
From~\eqref{equ:VX} it follows that
\begin{align*}
\mathbb{V}[X] &\leq \sum_{[d_1\!,\, d_2] \,\leq\, n} \varphi(d_1)\,\varphi(d_2) \, \beta^{\lfloor n / d_1 \rfloor + \lfloor n / d_2 \rfloor - \lfloor n / [d_1, d_2] \rfloor} \big(1 - \beta^{\lfloor n / [d_1, d_2] \rfloor} \big) \\
&\leq \sum_{[d_1\!,\, d_2] \,\leq\, n} d_1 d_2 \cdot \frac{\alpha n}{[d_1, d_2]} = \alpha n \sum_{[d_1\!,\, d_2] \,\leq\, n} (d_1, d_2) \leq \alpha n \sum_{d \,\leq\, n} d \sum_{a_1 a_2 \,\leq\, n / d} 1 \\
&= \alpha n \sum_{d \,\leq\, n} d \sum_{m \,\leq\, n / d} \tau(m) \ll \alpha n^2 \sum_{d \,\leq\, n} \log\!\left(\frac{n}{d}\right) = \alpha n^2 \big(n \log n - \log (n!)\big) < \alpha n^3 ,
\end{align*}
where we used Lemma~\ref{lem:Bernoulli}, Lemma~\ref{lem:tau}, and the bound $n! > (n / \mathrm{e})^n$.
Thus~\eqref{equ:variance_upper} is proved.
\end{proof}

\begin{proof}[Proof of Theorem~\ref{thm:Aq}]
By Chebyshev's inequality, Theorem~\ref{thm:expectation} and Theorem~\ref{thm:variance}, we have
\begin{equation*}
\mathbb{P}\big[\,|X - \mathbb{E}[X]| > \varepsilon\!\; \mathbb{E}[X] \,\big] < \frac{\mathbb{V}[X]}{\big(\varepsilon \mathbb{E}[X]\big)^2} \ll \frac{\alpha n^3}{(\varepsilon \alpha n)^2} \ll \frac1{\varepsilon^2 \alpha n} = o_\varepsilon(1) ,
\end{equation*}
as $\alpha n \to +\infty$.
Hence, using again Theorem~\ref{thm:expectation}, we get
\begin{equation*}
X \sim \frac{3}{\pi^2} \cdot \frac{\alpha \Li_2(1 - \alpha)}{1 - \alpha} \cdot n^2 ,
\end{equation*}
with probability $1 - o(1)$, as $\alpha n \to +\infty$.
\end{proof}

\bibliographystyle{amsplain}

\end{document}